\documentclass[
final
]{dmtcs-episciences}

\usepackage[utf8]{inputenc}
\usepackage{subfigure}

\usepackage{complexity}
\usepackage{amsmath}
\usepackage{amssymb}
\usepackage{enumerate}
\usepackage{hyperref}
\usepackage{amsthm}

\newtheorem{theorem}{Theorem}

\newtheorem{lemma}[theorem]{Lemma}
\newtheorem{corollary}[theorem]{Corollary}

\newtheorem{proposition}{Proposition}

\DeclareMathOperator*{\argmax}{arg\,max}
\DeclareMathOperator*{\argmin}{arg\,min}
\graphicspath{{figs/}}

\usepackage[round]{natbib}

\author{Helena Bergold\affiliationmark{1}
	\and 
	Winfried Hochstättler\affiliationmark{2}
	\and 
	Uwe Mayer\affiliationmark{2}}

\title[The Neighborhood Polynomial of Chordal Graphs]{The Neighborhood Polynomial \\ of Chordal Graphs\thanks{A short version of this paper appeared in the Proceedings of the 17th Algorithms and Data Structures Symposium (WADS 2021), see \cite{WADS}. \goodbreak 
		The authors thank Kolja Knauer and Manfred Scheucher for helpful discussions and the anonymous reviewers for helpful comments. \goodbreak
		Helena Bergold was supported by DFG-GRK 2434.}}
\affiliation{
	Freie Universit\"at Berlin, Department of Computer Science, 
	Germany  \\
	FernUniversit\"at in Hagen, Fakult\"at f\"ur Mathematik und Informatik,
	Germany}

\received{2021-08-25}

\revised{2022-04-22}

\accepted{2022-04-22}

\begin{document}
	\publicationdetails{24}{2022}{1}{19}{8388}
	\maketitle
	\begin{abstract}
		We study the neighborhood polynomial and the complexity of
		its computation for chordal graphs.
		The neighborhood polynomial of a graph is the generating function
		of subsets of its vertices that have a common neighbor.
		We introduce a parameter for chordal graphs called
		anchor width
		and an algorithm to compute the neighborhood polynomial which runs in
		polynomial time if the anchor width is polynomially bounded. The anchor width is the maximal number of different sub-cliques of a clique which appear as a common neighborhood.
		Furthermore we study the anchor width for chordal graphs and some subclasses such as chordal comparability graphs and chordal graphs with bounded leafage. The leafage of a chordal graphs is the minimum number of leaves in the host tree of a subtree representation. We show that the anchor width of a chordal graph is at most $n^{\ell}$ where $\ell$ denotes the leafage. 
		This shows that for some subclasses computing the neighborhood polynomial is possible in polynomial time while it is \NP-hard for general chordal graphs.
	\end{abstract}

	\section{Introduction}
	\label{sec:Introduction}
	
	In this paper we study the neighborhood polynomial of graphs and give an algorithm to compute the polynomial for chordal graphs in polynomial time for some subclasses. 
	Throughout the paper, all graphs are simple, finite and undirected.
	For a graph $G = (V,E)$, the \emph{neighborhood} of a vertex $v \in V$ is the set of all adjacent vertices, denoted by $N_G(v) = \{u \in V \mid uv \in E \}$. 
	The \emph{neighborhood complex} of a graph $G$, first introduced by \cite{Lovasz1978}, consists of all subsets of vertices $W \subseteq V$ which have a common neighbor, that is 
	\[ 
		 \mathcal{N}_G = \{U \subseteq V | \: \exists v \in V: U \subseteq N_G(v) \}.
	\]
	This set-system is clearly hereditary and hence it is a simplicial complex.
	To count the number of sets with cardinality $k$ in $\mathcal{N}_G$, we define the \emph{neighborhood polynomial}
	\[
		N_G(x) = \sum_{U \in \;\mathcal{N}_G} x^{|U|},
	\]
	which is the generating function of the neighborhood complex $\mathcal{N}_G$.
	Since we only consider finite graphs, the sum is finite and $N_G(x)$ is a polynomial such as all other generating functions considered in this paper.  
	We investigate the complexity of computing the neighborhood polynomial of some graph classes. In particular, we look at chordal graphs and subclasses like interval graphs, split graphs and chordal comparability graphs. 
	In order to do this, we introduce the \emph{anchor width} of a graph and develop an algorithm for computing the neighborhood polynomial in Section~\ref{sec:anchorDegree}. We will see that the anchor width is the essential parameter for a polynomial runtime of our algorithm. If for any subclass of chordal graphs the anchor width is polynomially bounded in the number of vertices, our algorithm is efficient. In particular our main result is the following theorem.
	\begin{theorem} \label{thm:runtime}
		Let $G$ be a chordal graph with $n$ vertices and anchor width $k$. Computing the neighborhood polynomial takes at most $\mathcal{O}(n^3 k + n^2k^2)$ time. 
	\end{theorem}
	In Section~\ref{sec:AnchorWidth} we investigate the complexity of the anchor width for different subclasses.
	For this we look at chordal graphs with bounded leafage. The leafage $\ell(G)$ was introduced in~\cite{LinMcKeeWest1998} and measures how close a chordal graph is to an interval graph. 
	We show that a chordal graph~$G$ on $n$ vertices has anchor width at most $n^{\ell(G)}$ (cf. Theorem~\ref{thm:leafage}). Furthermore for interval graphs, which are the graphs with leafage at most two, we give a family with quadratic anchor width. 
	Another result of Section~\ref{sec:AnchorWidth} is that chordal comparability graphs have linearly bounded anchor width.

	\section{Preliminaries}
	\label{sec:Preliminaries}
	
	The neighborhood polynomial was introduced by \cite{BrownNowakowski2008} who investigated the effect of some elementary graph operations on the neighborhood polynomial. Given two graphs $G_1=(V_1,E_1)$ and $G_2 = (V_2,E_2)$ on disjoint vertex sets, the \emph{union} $G_1 \cup G_2$ of the graphs is the graph on the vertex set $V_1 \cup V_2$ with edge set $E_1 \cup E_2$. 
	The \emph{join} $G_1 +G_2$ of the two graphs is the graph on the vertex set $V_1 \cup V_2$ consisting of both graphs together with all possible edges between vertices in $V_1$ and vertices in $V_2$, that is $E = E_1 \cup E_2 \cup \{ v_1v_2 \mid  v_1 \in V_1, v_2 \in V_2\}$.
	
	\begin{proposition}[\cite{BrownNowakowski2008}]
		\label{prop:DisjointUnion}
		Let $G_1$ and $G_2$ be two graphs on disjoint vertex sets.
		Then the neighborhood polynomial of the disjoint union $G_1 \cup G_2$ is
		\[
		N_{G_1 \cup G_2}(x) = N_{G_1}(x) + N_{G_2}(x) -1 .
		\]
	\end{proposition} 
	
	\begin{proposition}[\cite{BrownNowakowski2008}]
		\label{prop:Join}
		Let $G_1 = (V_1,E_1), G_2 = (V_2,E_2)$ be two graphs on disjoint vertex sets. 
		Then the neighborhood polynomial of the join $G_1 + G_2$ is
		\[
		N_{G_1+ G_2}(x) = (1+x)^{|V_2|}N_{G_1}(x) + (1+x)^{|V_1|}N_{G_2}(x)- N_{G_1}(x )N_{G_2}(x).
		\]
	\end{proposition}
	
	The two graph operations, disjoint union and join, are used to define cographs. \emph{Cographs} are exactly the graphs which do not contain an induced $P_4$, a path on four vertices. They can be constructed recursively. Starting with a single vertex as a cograph, the disjoint union and the join of two cographs are cographs.
	For this and other well-known graph theoretic facts, we refer to~\cite{Golumbic1980}.
	The neighborhood polynomial of a single vertex graph is $N(K_1,x) =1$ and the two operations disjoint union and join, given by the two formulas in Proposition~\ref{prop:DisjointUnion} and Proposition~\ref{prop:Join} are computable in linear time. Note that $(1+x)^n$ can be computed in linear time using the binomial theorem. \cite{CorneilPerlStewart1985} present a linear time algorithm to recognize cographs and give the corresponding recursive construction rules using disjoint union and join. Hence the neighborhood polynomial of a cograph is computable in quadratic time.
	
	Another graph operation is attaching one vertex $v$ to a subset of vertices of a graph~$G$. This operation was studied by \cite{AlipourTittmann2021}, who gave an explicit formula for a neighborhood polynomial after attaching a vertex to a subset of vertices. More formally for a graph $G = (V,E)$, a subset $U \subseteq V$ of vertices and an additional vertex $v \notin V$, we denote by $G_{U \triangleright v}$ the graph with vertex set $V \cup \{v\}$ and edge set $E \cup \{uv \mid u \in U\}$. 
	For simplification we use the following notation for all $W\subseteq V$:
	\begin{align*}
		N_G^{\cap}(W) &= \bigcap_{w \in W} N_G(w) \quad \text{and} \\
		N_G^{\cup}(W) &= \bigcup_{w \in W} N_G(w).
	\end{align*}
	
	\begin{proposition}[\cite{AlipourTittmann2021}]
		\label{prop:AlipourTittmannUpdate}
		Let $G = (V,E)$ be a graph, $U \subseteq V$ and $v \notin V$. 
		Then the neighborhood polynomial of $G_{U \triangleright v}$ is
		\[
		N_{G_{U \triangleright v}}(x) = N_G(x) + \sum\limits_{\substack{W \subseteq U, \\W \neq \emptyset}} \phi_W + \sum\limits_{\substack{W \subseteq U, \\ W \neq \emptyset}} (-1)^{|W|+1} x(1+x)^{|N_G^{\cap}(W)|},
		\]
		where 
		\[
		\phi_W = 
		\begin{cases}
		x^{|W|},&\text{ if } N_G^{\cap}(W) = \emptyset; \\
		0,&\text{ otherwise.}
		\end{cases}
		\]
	\end{proposition}
	
	Using this formula, Alipour and Tittmann showed that for a fixed integer~$k$, computing the neighborhood polynomial of $k$-degenerate graphs is possible in polynomial time, see~\cite{AlipourTittmann2021}.
	A \emph{$k$-degenerate graph} is a graph where every subgraph has a vertex~$v$ with $\deg(v) \leq k$.
	Using the degeneracy, we can pick one vertex of degree $\leq k$ after another and update the neighborhood polynomial by the formula of Proposition \ref{prop:AlipourTittmannUpdate} in order to get a polynomial runtime. 
	As a corollary it follows that there is a polynomial-time algorithm to compute the neighborhood polynomial for planar (or more general graphs of bounded genus) and $k$-regular graphs, see \cite{AlipourTittmann2021}.
	This update formula of Alipour and Tittman (see Proposition~\ref{prop:AlipourTittmannUpdate}) was the starting point of our investigations for chordal graphs. 
	
	A graph $G$ is said to be \emph{chordal} if there is no induced cycle of length $\geq 4$.
	Equivalently a graph is chordal if and only if it has a perfect elimination order.
	A \emph{perfect elimination order} is an ordering $v_1, \ldots, v_n$ of the vertices such that for all $i$ the neighborhood of $v_i$ in $G[\{v_i, \ldots v_n \}]$ is a clique. 
	Here for a subset $U \subseteq V$ the graph $G[U]$ denotes the subgraph of $G$ induced by $U$.
	A vertex, whose neighborhood is a clique is called \emph{simplicial}.
	It is well-known that every non-trivial chordal graph has at least two simplicial vertices, which gives us the perfect elimination order (cf. \cite{Golumbic1980}).
	In order to study the neighborhood polynomial of chordal graphs and their subclasses, we make use of the perfect elimination order to build the chordal graph by attaching one vertex after another to a clique. 
	We adapt the formula of Alipour and Tittmann (Proposition~\ref{prop:AlipourTittmannUpdate}) to our use. 
	To get some complexity results of computing the neighborhood polynomial, the connection to the domination polynomial is useful. For this we introduce dominating sets.
	A \emph{dominating set} of a graph $G = (V,E)$ is a set of vertices $D \subseteq V$ such that every vertex of the graph is either in $D$ or adjacent to a vertex of $D$, i.e.  
	\begin{align*}
	D \ \cup N^{\cup}_G(D) = V.
	\end{align*}
	The family of all dominating sets of a graph $G$ is denoted by $\mathcal{D}_G$ and the \emph{domination polynomial} $D_G(x)$ is the generating function of $\mathcal{D}_G$ that is 
	\begin{align*}
	D_G(x) = \sum_{U \in \mathcal{D}_G} x^{|U|}.
	\end{align*}
	The following relation between domination polynomials and neighborhood polynomials holds. For a proof see for example \cite{HeinrichTittmann2018}. 
	\begin{proposition} \label{prop:connectionDominating}
		For a graph $G = (V,E)$ and its complement graph $\overline{G}$ it holds:
		\begin{align*}
		D_{\overline{G}}(x) + N_G(x) = (1+x)^{|V|}.
		\end{align*}
	\end{proposition}
	With other words this proposition states that every vertex set has either a common neighbor in the graph or is a dominating set in the complement graph. 
	The connection of these two polynomials can be used to determine the complexity of computing the neighborhood polynomial. In particular, the neighborhood polynomial is computable in polynomial time if and only if the domination polynomial of the complement graph is computable in polynomial time. 
	Furthermore the contributions to the well-known graph problem DOMSET, imply some complexity results for the neighborhood polynomial. DOMSET is the problem of
	deciding whether a graph has a dominating set of size $\leq k$ for a given $k$. 
	\begin{corollary} \label{col:domsethard}
		Let $\mathcal{G}$ be a class of graphs and $\overline{\mathcal{G}}$ the class of the complement graphs of $\mathcal{G}$. 
		If DOMSET is \NP-complete on $\overline{\mathcal{G}}$, then computing the neighborhood polynomial on $\mathcal{G}$ is \NP-hard.
	\end{corollary}
	
	DOMSET is \NP-hard on  
	many graph classes such as chordal graphs, see~\cite{BoothJohnson1982}. 
	\cite{Bertossi1984} showed that it is \NP-hard on bipartite graphs and split graphs. \emph{Split graphs} are the graphs where the vertex set can be partitioned into a clique and an independent set. Since split graphs are exactly the graphs which are chordal and \emph{co-chordal} (i.e. the complement graph is chordal), DOMSET is also \NP-hard on co-chordal graphs. 
	This together with Corollary~\ref{col:domsethard} shows the \NP-hardness of computing the neighborhood polynomial in split graphs (cf.~\cite{Day2017}) and hence in chordal graphs.
	
	\section{Algorithm for Chordal Graphs}
	\label{sec:anchorDegree}

	Our algorithm relies on the perfect elimination order of chordal graphs and comes from the vertex-attachment formula of Alipour and Tittmann, see Proposition~\ref{prop:AlipourTittmannUpdate}. 
	First, we adapt their formula to our special case where we attach a vertex to a clique. 
	To study the new arising neighborhood sets after vertex attachment, we introduce anchor sets, which are subsets of a clique appearing as a common neighborhood of a set of vertices. 
	The maximal number of anchor sets of a clique, which we denote as anchor width, is the essential parameter in our algorithm to get a polynomial runtime. \\

	\begin{figure}[tb]
		\centering
		\includegraphics[scale = 0.6,page = 1]{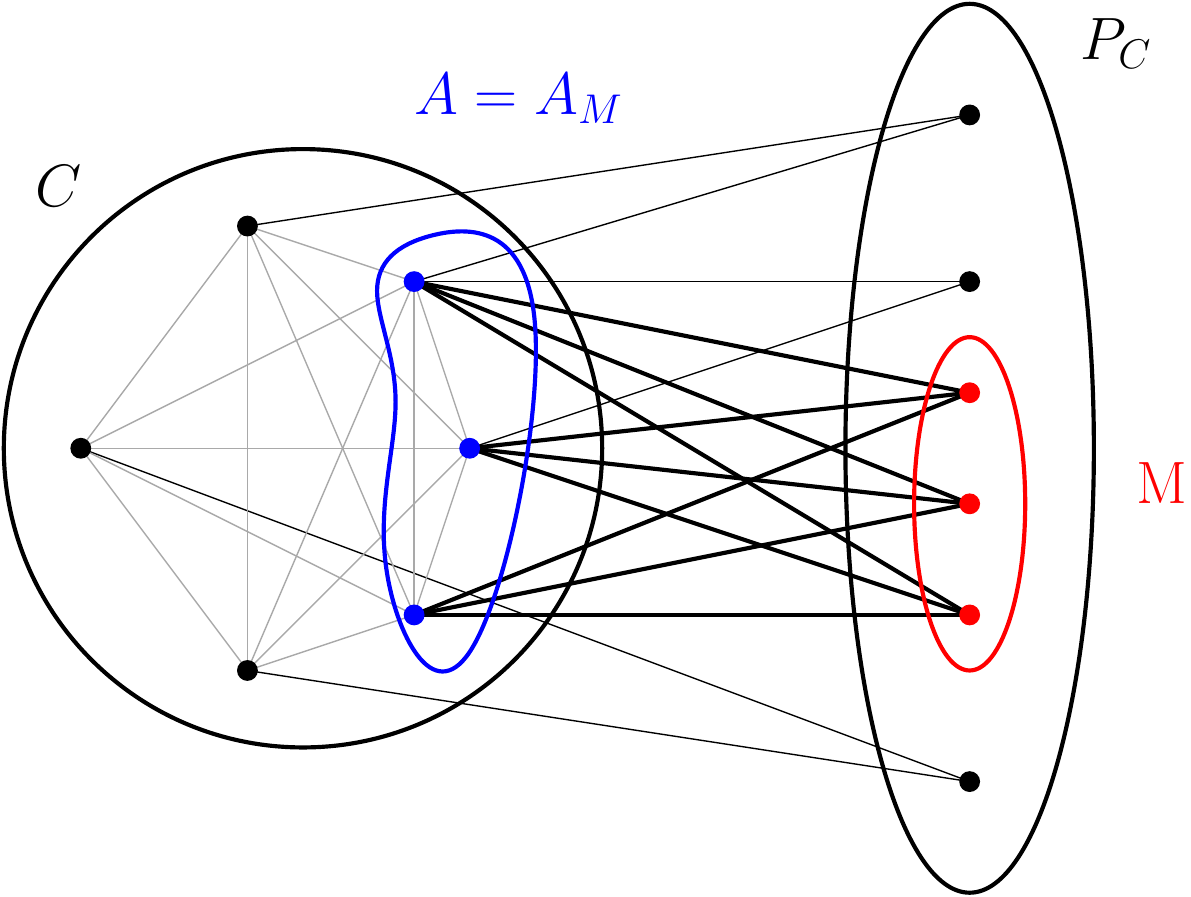}
		\caption{An illustration of the introduced sets, the periphery $P_C$ and a periphery set $M$ with corresponding anchor set $A_G(M,C)$ of a clique $C$. }
		\label{fig:illustrationDefinition}
	\end{figure}
	
	Let $C$ be a clique in a graph $G = (V,E)$.
	We define the set of neighbors of the clique $C$, not including the clique itself as the \emph{periphery} of $C$, denoted by 
	\begin{align*}
	P_G(C) = N^{\cup}_G(C) \backslash C.
	\end{align*}
	A subset $M \subseteq P_G(C)$ of the periphery is called \emph{periphery set}. Note that the empty set is also a periphery set. 
	We call a non-empty subset $A$ of $C$ \emph{anchor set}, if it is the common neighborhood in $C$ of some periphery set $M$. See Figure~\ref{fig:illustrationDefinition} for an illustration. In general not all subsets of a clique are an anchor set. 
	For every $M \subseteq P_G(C)$ we define the \emph{corresponding anchor set} in $C$ as
	\begin{align*}
	A_G(M,C) =  N^{\cap}_G(M) \cap C	
	\end{align*}
	if the intersection is non-empty. The neighborhood of the empty periphery set $M = \emptyset$ is the empty intersection and thus the corresponding anchor set is $C$. 
	Moreover, several periphery sets $M$ and $M'$ can correspond to the same anchor set $A_G(M,C) = A_G(M',C)$.
	For an anchor set $A$ of $C$, the periphery sets $M \subseteq P_G(C)$ whose common  neighborhood in $C$ is $A$ build the family
	\[
	\mathcal{P}_G(A,C) = \left\{M \subseteq P_G(C) \mid A_G(M,C) = A \right\}.
	\]
	The generating function of $\mathcal{P}_G(A,C)$ is called \emph{periphery polynomial} and defined by
	\[
	P_G(A,C,x) = \sum_{M \in \mathcal{P}_G(A,C)} x^{|M|}.
	\]
	Note that $\mathcal{P}_G(A,C) = \emptyset$ and $P_G(A,C,x) =0$ if $A$ is not an anchor set of $C$. 
	The family of all anchor sets of a clique $C$ is 
	\[
	\mathcal{A}_G(C) = \{A \subseteq C \mid  A \neq \emptyset \text{ and } \exists M \subseteq P_G(C):  A = A_G(M,C)\}.
	\]
	Note that $C \in  \mathcal{A}_G(C)$ for every clique $C$, since $C$ is the anchor set of the empty periphery set. 
	The \emph{anchor width} of a graph $G$ is the smallest number $k$ such that $|\mathcal{A}_G(C) | \leq k$ for all cliques $C$ in $G$. 
	
	For a maximal clique $C_{max}$ and a clique $C$ contained in $C_{max}$ the following relations hold.
	
	\begin{lemma}\label{lem:InclusionClique}
		Let $C$ be a clique and $C_{max}$ a maximal clique containing $C$ in a graph $G= (V,E)$.
		Then the following conditions hold:
		\begin{enumerate}[(a)]
			\item $C_{max}\backslash C \subseteq P_G(C) \subseteq P_G(C_{max}) \cup (C_{max} \backslash C)$ 
			\item $\mathcal{A}_G(C) = \{A \cap C \mid A \cap C \neq \emptyset \text{ and }A \in \mathcal{A}_G(C_{max})\} $ 
			\item For every $A \in \mathcal{A}_G(C)$ the periphery polynomial is
			\[
			P_G(A,C,x)  = (1+x)^{|C_{max}\backslash C|}\sum_{ \substack{A' \in \mathcal{A}_G(C_{max})\\ A' \cap C = A}}  P_G(A',C_{max},x).
			\] 
		\end{enumerate}
	\end{lemma}

\begin{proof}
	\textit{(a)} Since every vertex in the clique $C_{max} \backslash C$ is adjacent to~$C$, the first inclusion holds. 
	Furthermore, every element which is adjacent to one of the elements in $C$ is either an element of $C_{max} \backslash C$ or it is adjacent to an element of~$C_{max}$.
	
	To show \textit{(b)}, we check which subsets of the clique $C$ can appear as an anchor set. Let $M' \subseteq P_G(C)$ be a periphery set such that $A_G(M,C)$ is a non-empty anchor set. Using \textit{(a)} we distinguish three cases. 
	
	If $M' \subseteq C_{max} \backslash C$, the anchor set of $M'$ is $C$ itself. 
	If $M' \subseteq P_G(C_{max})$ there exists an anchor set $A = A_G(M,C_{max})$. Since $M'$ only consists of elements of the periphery of $C$, the intersection of $A$ with $C$ provides the anchor set $A_G(M,C)$, which is non-empty.
	In the final case, $M'$ consist of elements of $C_{max} \backslash C$ and $P_G(C_{max})$. We only need to consider $M = M' \cap P_G(C_{max})$, since the elements of $C_{max} \backslash C$ just lead to another intersection with $C$. We continue as in the second case. 
	
	On the other hand, a set $A' = A \cap C \neq \emptyset $ for $ A \in \mathcal{A}_G(C_{max})$ is always an anchor set. 
	For $A$ there exists a periphery set $M \subseteq P_G(C_{max})$ which has $A$ as common neighborhood $A_G(M,C_{max})$ in $C_{max}$. If we take all elements of $M$ which are in the periphery of $C$, the common intersection of those elements inside $C$ is exactly $A'$. Hence $A'$ is an anchor set. 
	
	\textit{(c)} The periphery polynomial $P_G(A,C,x)$ counts the different periphery sets with respective to the size, where $A$ is the corresponding anchor set. 
	As we have seen in \textit{(b)}, an anchor set $A$ is given by $A = A' \cap C$ for an anchor set $A' \in \mathcal{A}_G(C_{max})$. Since there are different possibilities to choose $A'$, we sum over all corresponding periphery polynomials which are counted in $P_G(A',C,x)$. 
	Furthermore all elements in $C_{max }\backslash C$ are in the periphery of $C$ (see (a)) with neighborhood $C$. 
	Hence we can add elements of $C_{max} \backslash C$ to any periphery set $M'$ with anchor set $A_G(M',C_{max}) = A'$ and still have as anchor set $A_G(M,C) = A$.
	For the polynomial as generating function, we multiply $P_G(A',C,x)$ by \mbox{$(1+x)^{|C_{max}\backslash C|}$}.
	
	For different anchor sets $A'$ and $A''$ with $A'' \cap C = A = A' \cap C$, the corresponding periphery sets are pairwise different, since the common neighborhood inside $C_{max}$ is different. 
	So in order to get all periphery sets corresponding to $A$, we need to sum up these the polynomials.
\end{proof}
	
	This shows that it is sufficient to provide the information about anchor sets and periphery polynomials for all maximal cliques of the graph. With this information we are able to compute the necessary information for all other cliques. Furthermore the anchor width only depends on the size of the anchor family of the maximal cliques.
	\medskip
	
	In the following, we derive a formula for the neighborhood polynomial after vertex attachment using the periphery polynomial. 
	For every set $U \subseteq V$ of vertices we define the \emph{local neighborhood} $\mathcal{N}_G(U)$ of $U$ as the family 
	consisting of all vertex sets of $G$ which have a common neighbor in~$U$, that is
	\[
	\mathcal{N}_G(U) = \{ W \subseteq V \mid \exists v \in U: \ W \subseteq N_G(v) \}.
	\]
	Note that the local neighborhood $\mathcal{N}_G(V)$ of the whole vertex set is equal to the neighborhood complex $\mathcal{N}_G$.
	For every clique $C$, we can partition the local neighborhood $\mathcal{N}_{G}(C)$ by the following lemma into the disjoint sets 
	\begin{align*}
	\mathcal{N}_G(A,C) = \{N \in \mathcal{N}_G(C) \mid  N \cap P_G(C) \in \mathcal{P}_{G}(A,C) \}, \quad A \in \mathcal{A}_G(C).
	\end{align*}

	\begin{lemma}\label{lem:generatingfunctNcG1}
		For every clique $C$ of the graph $G$, it holds
		\[
		\mathcal{N}_G(C) = \mathop{\dot \bigcup}_{A \in \mathcal{A}_G(C)} \mathcal{N}_G(A,C).
		\]
	\end{lemma}
	\begin{proof}
		For every $N \in \mathcal{N}_G(C)$ there is by definition a clique-vertex $v \in C$ which is adjacent to every element in $N$. Hence the common neighborhood of $N \cap P_G(C)$ inside $C$ is non-empty. Note that in general $N$ is not a periphery set since it can contain elements from $C$. The common neighborhood is an anchor set $A$. Since these anchor sets differ for different families $\mathcal{N}_G(A,C)$ the union is disjoint.
	\end{proof}
	
	Lemma \ref{lem:generatingfunctNcG1} is useful since we only have to determine the generating functions of $\mathcal{N}_G(A,C)$ for every $A \in \mathcal{A}_G(C)$. Adding these generating functions, we maintain the generating function of the local neighborhood $\mathcal{N}_G(C)$. 
	In the next lemma, we derive a formula to compute the generating function of
	$\mathcal{N}_{G}(A,C)$  for every $A \in \mathcal{A}_G(C)$. 
	
	\begin{lemma}\label{lem:generatingfunctNcG2}
		For a given anchor set $A \in \mathcal{A}_G(C)$ of a clique $C$, the generating function of $\mathcal{N}_{G}(A,C)$ is
		\[
		N_G(A,C,x) = P_G(A,C,x) \left( (1+x)^{|C|}- x^{|A|}(1+x)^{|C|-|A|} \right).
		\]
	\end{lemma}
	\begin{proof} 
		We count the number of sets with respect to the cardinality in $\mathcal{N}_{G}(A,C)$. 
		Every $M \in \mathcal{P}_G(A,C)$ is in $\mathcal{N}_{G}(A,C)$. Hence $P_G(A,C,x)$ must be a summand of $N_G(A,C,x)$.
		Furthermore there are supersets $N$ for all $M$ which contribute to $N_G(A,C,x)$. 
		Since we look at all $M \in \mathcal{P}_G(A,C)$, it is enough to look at supersets $N = M \cup X$, where $X$ is a subset of $C$. 
		In order to keep $N$ in the local neighborhood $\mathcal{N}_G(C)$, we need a common neighbor in $C$. Since the common neighborhood of $M$ inside $C$ is the anchor set $A$, the common neighborhood of $N$ must contain an element of $A$. Hence $X$ cannot be the whole anchor set $A$. 
		In particular, the possibilities to extend $M$ are the elements of the family
		\[	
		\mathcal{X} = \{X \mid \exists a \in A: X \subseteq C \backslash \{a\}  \}.
		\]
		All sets in $\mathcal{X}$ consist of a disjoint union of a proper subset of $A$ and a subset of $C \backslash A$. This leads to the generating function 
		\[
		\left( (1+x)^{|A|} - x^{|A|} \right)(1+x)^{|C|-|A|}  
		= (1+x)^{|C|}- x^{|A|}(1+x)^{|C|-|A|}
		\]
		of $\mathcal{X}$. 
		The generating function of $\mathcal{P}_G(A,C)$, which counts the different possibilities of $M$ is counted by $P_G(A,C,x)$.
	\end{proof}
	
	We are now ready to prove the update formula for the neighborhood polynomial after attaching a vertex to a clique. The proof and the formula are similar to Proposition~\ref{prop:AlipourTittmannUpdate} (cf.~\cite{AlipourTittmann2021}).
	\begin{proposition} \label{cor:AttachmentClique}
		Let $G = (V,E)$ be a graph and $C$ a clique in the graph. 
		The neighborhood polynomial of $G_{C \triangleright v}$ with vertex set $V \cup \{v\}$ is:
		\begin{align*}
		N_{G_{C \triangleright v}}(x) &= N_G(x)+\phi_G(C) \\
		&
		+ \ x{\sum\limits_{A \in \mathcal{A}_G(C)}}{P_G(A,C,x) \left( (1+x)^{|C|}- x^{|A|}(1+x)^{|C|-|A|} \right)},
		\end{align*}
		where
		\[
		\phi_G(C) = \begin{cases}
		x^{|C|}, &\text{ if } C \text{ is a maximal clique in } G;\\
		0, &\text{ otherwise }.
		\end{cases}
		\]
	\end{proposition} 
	
	\begin{proof}
		Let $X \in \mathcal{N}_{G_{C \triangleright v}}$ be a neighborhood set of the graph $G_{C \triangleright v}$. We consider the following three cases:
		\begin{itemize}
			\item If $X \subseteq V$ and $X \not \subseteq C$, then $X\in \mathcal{N}_G$ is in the neighborhood complex of $G$. Hence $X$ is considered in the first summand $N_G(x)$ of the above formula.
			\item Now let $X \subseteq V$ and $X \subseteq C$. If $X$ is a proper subset of $C$, it already has a common neighbor in $G$, hence it is already counted in the first summand. 
			Similarly this holds if $X = C$ and $C$ is not a maximal clique in $G$.
			Thus the only case where a new neighborhood arises is if $X = C$ and $C$ is a maximal clique in $G$.
			In $G_{C \triangleright v}$ the common neighbor of $C$ is $v$. This is counted in the summand in $\phi_G(C)$. 
			\item Let us now consider the case $v \in X$, i.e. $ X \not \subseteq V$. 
			Since $v$ is connected to all elements in $C$, we need to count all subsets $Y \subseteq V$ which have a common neighbor in $C$. This is equivalent to count the number of elements in $\mathcal{N}_G(C)$. Combining Lemma \ref{lem:generatingfunctNcG1} and Lemma \ref{lem:generatingfunctNcG2}, we obtain 
			\[
			\sum_{A \in \mathcal{A}_G(C)} P_G(A,C,x) \left( (1+x)^{|C|}- x^{|A|}(1+x)^{|C|-|A|} \right)
			\]
			as the generating function of $\mathcal{N}_G(C)$. In $X$ there is one additional element~$v$. Hence we have to multiply the polynomial with $x$.
		\end{itemize}
		Since the above cases are disjoint, this leads to the formula of the neighborhood polynomial as stated. 
	\end{proof}
	
	With this formula, we are able to compute the neighborhood polynomial after attaching a vertex $v$ to a clique $C$ in a graph. We will use this in connection with the perfect elimination order of the graph to compute the neighborhood polynomial of chordal graphs. For this we start with a single vertex and add in reverse order of the perfect elimination order the vertex one after another. After every step we compute the neighborhood polynomial using Proposition~\ref{cor:AttachmentClique}. Furthermore we need to make sure to have the correct data to compute the neighborhood polynomial after every step. Hence we need to update the anchor family and periphery polynomials. 
	
	To go into a more detailed analysis we assume that the chordal graph is connected. 
	If the chordal graph is not connected we apply the algorithm for every connected component and compute the neighborhood polynomial of the whole graph by adding the neighborhood polynomials of every connected component (cf. Proposition~\ref{prop:DisjointUnion}). This operation can be done in linear time. 
	
	For a connected chordal graph we compute the elimination order $v_1, \ldots , v_n$ using a lexicographic breadth-first search which has linear runtime. 
	We then add the vertices in reverse order, starting with $v_n$. The neighborhood polynomial of this starting graph is $1$. Then we add $v_i$ to the corresponding clique in $G[v_{i+1}, \ldots, v_n]$ and compute the new neighborhood polynomial with Proposition~\ref{cor:AttachmentClique}.

	In order to compute the formula of Proposition~\ref{cor:AttachmentClique}, we need the anchor family of $C$ and the corresponding periphery polynomials $P_G(A,C,x)$ for every $A \in \mathcal{A}_G(C)$. 
	With Lemma~\ref{lem:InclusionClique} it suffices to store these informations for the maximal cliques of the graph and compute them in every step for the required clique $C$ where we attach the new vertex. 
	The details of updating this information will be explained in the next paragraph. Furthermore we have a list of the maximal cliques in the current graph which need to be updated, too. Since every step gives at most one new maximal clique, there are at most $n$ maximal cliques in a chordal graph with $n$ vertices. 
	 After attaching one vertex, updating the polynomial and the corresponding information, we go on with the next step.  \\
	
	We now study how to update the anchor families and periphery polynomials for the maximal cliques after attaching a vertex in order to have the correct ones in the next step. 
	Fix a clique $C$ of the graph $G$. The graph with attached vertex $v$ to $C$ is denoted by $G^+ = G_{C \triangleright  v}$.   
	We get exactly one new maximal clique $C^+ = C \cup \{v\}$ which we have to add to the list of maximal cliques in the graph. If $C$ is a maximal clique in $G$ we have to delete the entry $C$ from the list of maximal cliques. 
	
	We determine the anchor sets and periphery polynomial of the new arising maximal clique $C^+$.
	The periphery of $C^+$ in $G^+$ is $P_{G^+}(C^+) = P_G(C)$ and the family of anchor sets is $\mathcal{A}_{G^+}(C^+) = \mathcal{A}_{G}(C) \cup \{ C^+\}$ if $C$ was not maximal in $G$ and $\mathcal{A}_{G^+}(C^+) = \mathcal{A}_{G}(C) \backslash \{C\} \cup \{ C^+\}$ if $C$ was maximal in $G$. The periphery polynomials stay as in $G$ that is $P_{G^+}(A,C^+,x) = P_G(A,C,x)$ for all $A \in \mathcal{A}_{G}(C)$ and $P_{G^+}(C^+,C^+,x) = 1$.
	
	Now we go through the list of maximal cliques and update the necessary information if needed. 
	The maximal cliques in $G$ which have no intersection with $C$, do not change in~$G^+$.
	Let $C_{max}$ be a maximal clique in $G$ with $C_{max} \cap C \neq \emptyset$. It is $C \neq C_{max}$. 
	Note that we deleted $C$ from the list in the case that $C$ was maximal in $G$. 
	The periphery of $C_{max}$ in $G^+$ consists of the periphery of $C_{max}$ in $G$ together with the new element $v$. More formally it holds
	\begin{align*}
	P_{G^+}(C_{max})=P_G(C)\cup \{v\}.
	\end{align*}
	Every anchor set of~$C_{max}$ in $G$ remains an anchor set in~$G^+$. 
	Since the new vertex $v$ is attached to the subset $C_{max} \cap C$ of the considered clique $C_{max}$, this subset $C_{max} \cap C \neq \emptyset$ is a new anchor set in $G^+$, if it was not already an anchor set in $G$. 
	Furthermore all subsets of $C_{max}$ occurring as non-empty intersection of $C_{max} \cap C$ with an anchor set in $\mathcal{A}_G(C_{max})$ build an anchor set of $C_{max}$ in~$G^+$. 
	The family of anchor sets of $C_{max}$ in $G^+$ is:
	\begin{align} \label{eq:newanchor}
	&\mathcal{A}_{G^+}(C_{max}) = \\
	&\mathcal{A}_G(C_{max}) 
	\cup \{A \cap (C_{max} \cap C) \mid A \cap (C_{max} \cap C) \neq \emptyset \text{ and } A \in \mathcal{A}_G(C_{max}) \}. \nonumber
	\end{align}
	Note that  $C_{max} \cap C$ is an element of the second set in the above equation since $C_{max} \in \mathcal{A}_G(C_{max})$.
	To determine the periphery polynomial $P_{G^+}(A,C_{max},x)$ for every anchor set $A \in \mathcal{A}_{G^+}(C_{max})$, we need to distinguish the following three cases. 
	Since $C \neq C_{max}$, the intersection $C_{max} \cap C $ is a proper subset of $C_{max}$.
	\begin{enumerate}[(i)]
		\item If $A$ is a proper subset of $C \cap C_{max}$, all corresponding periphery sets in $G$ are a corresponding periphery set in $G^+$ and we can add $v$ to every corresponding periphery set $M$ in $G$, since the intersection with the neighborhood $N_{G^+} (v) = C$ does not change the anchor set. In this case we get:
		\begin{align*}
					P_{G^+}(A,C_{max},x) = (1+x) P_G(A,C_{max},x).
		\end{align*}
		
		\item If $A = C \cap C_{max}$, the periphery sets in $G$ with corresponding anchor set $A$ which are counted in $P_G(A , C_{max} , x)$ still have the same anchor set in $G^+$. 
		Furthermore $v$ is a new periphery set with anchor set $A = C \cap C_{max}$ and all periphery sets which have a superset $A'$ of $A$ as corresponding anchor set, form together with $v$ a periphery set with anchor set $A$. 
		Hence the updated periphery polynomial is	
		\begin{align*}
		&P_{G^+}(A , C_{max} , x) = 
		P_G(A , C_{max} , x) + x \left(1 + \sum_{\substack{A' \supseteq A \\ A' \in \mathcal{A}_{G}(C_{max})} } P_G(A',C_{max},x) \right).
		\end{align*}
		\label{item:updateperiphery}
		\item 
		In the remaining case $A$ is not a subset of $C \cap C_{max}$, hence $v$ is not in a periphery set with anchor set $A$. So the periphery polynomial stays the same, which means 
		\[
		P_{G^+}(A, C_{max},x) = P_G(A,C_{max},x).
		\]
	\end{enumerate}
	
	This concludes the analysis of the cases and hence the algorithm to compute the neighborhood polynomial of a chordal graph.
	This algorithm leads to a polynomial time algorithm if the anchor width is polynomially bounded as stated in Theorem~\ref{thm:runtime}. 
	We now give a detailed analysis of the runtime which concludes the proof of Theorem~\ref{thm:runtime}.
	
	\begin{proof}[of Theorem~\ref{thm:runtime}]
			Using lexicographic breadth-first search, we get a perfect elimination order of the chordal graph $G$ in linear time (cf.~\cite{Golumbic1980}).
			After attaching one vertex, there is one new maximal clique, containing~$v$. We add this maximal clique to the list of all maximal cliques in the graph, possibly removing the neighborhood of $v$ if it was in the list. 
			This shows that we have at most $n$ maximal cliques. 
			Furthermore we have an ordering of the vertices given by the perfect elimination order which is the reverse order of attaching the vertices to the graph. 
			Hence in order to compute an intersection or test a subset relation, we only need to compare the elements in the clique in the attached order. Both is possible in linear time.
			
			For every step we need to update the neighborhood polynomial and afterwards the other informations. In total we attach $n$ vertices and hence have $n$ steps. 
			
			To update the neighborhood polynomial after vertex attachment to a clique $C$, we first need to find a maximal clique, containing $C$ and compute its anchor family and its periphery polynomials. 
			In order to find this maximal clique we test for every maximal clique (at most $n$) whether $C$ is a subset. This needs at most $\mathcal{O}(n^2)$ time.
			
			Computing the anchor family takes $\mathcal{O}(nk)$ and for one anchor set the periphery polynomial takes $\mathcal{O}(nk)$ and since we have at most $k$ anchor sets in a clique, we need $\mathcal{O}(nk^2)$ for this step (cf. Lemma~\ref{lem:InclusionClique}). 
			Now using the update formula to compute the neighborhood polynomial is possible in $\mathcal{O}(k)$.
			Updating the neighborhood polynomial takes for every step at most $O(nk^2)$ time.
			
			Afterwards we update the anchor families and periphery polynomials for all maximal cliques with non-empty intersection with $C$. 
			We have at most $n$ maximal cliques, where we need to update this information. 
			Updating the anchor family for one maximal clique takes $\mathcal{O}(nk)$ time (see (\ref{eq:newanchor}))
			Hence for all maximal cliques this takes at most $\mathcal{O}(n^2k)$.
			For one anchor set, updating the periphery polynomials takes at most $\mathcal{O}(nk)$ in case~(\ref{item:updateperiphery}). This case appears at most once and compute the periphery polynomial in the other cases is possible in linear time. 
			Updating the necessary information for anchor sets and their periphery polynomials leads to a runtime of at most $\mathcal{O}(n^2k)$ for all maximal cliques. 
			Since we add in total $n$ vertices, one after another, this leads to a total runtime of $\mathcal{O}(n^3k + n^2k^2 )$. 
	\end{proof}
	
	\section{Complexity of the  Anchor Width} 
	\label{sec:AnchorWidth}
	In this section, we will discuss some subclasses of chordal graphs and study their anchor width. We show that there are subclasses with polynomially bounded anchor width. We arrived at these graph classes starting from interval graphs, the first class for which we found a polynomial bound.
	For these subclasses the algorithm explained in Section~\ref{sec:anchorDegree} runs in polynomial time.
	In contrast to this result, we show that the anchor width of split graphs, a simple well-known subclass, is not polynomially bounded (see the following proposition). Hence the algorithm introduced in Section \ref{sec:anchorDegree} might take super-polynomial time. \\

	\begin{proposition} \label{prop:splitexponential}
		For all $n \in \mathbb{N}$ there is a split graph on $n$ vertices with anchor width at least $2^{\frac{n}{2}}-1$. 
	\end{proposition}
	
	\begin{proof}
		We construct an infinite family of split graphs $S_m$ on $n = 2m $ vertices such that the anchor width is $2^m-1$.
		We start with a clique $C = \{c_1, \ldots , c_m\}$ of size~$m$ and attach vertices $v_1, \ldots, v_m$  such that every $v_i$ ($1 \leq i \leq m$) is adjacent to $c_j$ for all $j \neq i$.
		This constructed graph is a split graph since $C$ is a clique and $\{v_1, \ldots, v_m \}$ an independent set.
		All vertices $v_i$ are in the periphery $P_C$, hence $|P_C| = m= \frac{n}{2}$.
		For a periphery set $M \subseteq P_C$ we denote by $I_M$ the indices of those vertices $v_i$ which are in $M$. 
		The corresponding anchor set is then $\{C_j \mid j \notin I_M \}$. In such a way we can construct every non-empty subset of $C$ as an anchor set. 	
		Hence the anchor width of $S_m$ is $2^m -1 = 2^{\frac{n}{2}}-1$. 
	\end{proof}

	Another interesting family of subclasses are the chordal graphs with bounded leafage. For those graphs we can show a polynomial upper bound of the anchor width. 
	The leafage is a parameter which stems from the intersection graph representation of chordal graphs.
	An \emph{intersection graph} of a family of sets $\mathcal{F}$ is the graph consisting of one vertex for every set in $\mathcal{F}$ such that two vertices are adjacent if and only if the corresponding sets have a non-empty intersection. 
	An interval graph is an intersection graph of a family of subtrees of a path and
	chordal graphs are exactly the graphs which are the intersection graph of a family of subtrees of a \emph{host tree}~\cite{Gavril1974}.
	We call a representation of a chordal graph by a family of subtrees a \emph{subtree representation}.
	\cite{LinMcKeeWest1998} introduced the leafage of a chordal graph, which measures how close a chordal graph is to an interval graph.
	More precisely, the \emph{leafage} $\ell(G)$ of a chordal graph $G$ is defined as the minimum number of leaves of the host tree among all subtree representations of $G$. 
	We call a subtree representation \emph{optimal} if it has the minimum number of leaves in the host tree.
	The interval graphs are exactly the chordal graphs with leafage at most $2$.
	The split graphs $S_m$ constructed in Proposition~\ref{prop:splitexponential} have leafage $m$ and a possible optimal host tree is the star $K_{1,m}$.
	\cite{HabibStacho2009} present a polynomial-time algorithm in order to compute the leafage of a chordal graph. 
	As mentioned in \cite{LinMcKeeWest1998} we may restrict to host trees whose number of vertices is the number of maximal cliques of $G$. 
	\begin{lemma}
		\label{lem:1-1corresondence}
		There exists an optimal representation such that the vertices of the host tree are in one-to-one correspondence with the maximal cliques of the graph. 
	\end{lemma}
	\begin{proof}
		Since every pairwise intersecting family of subtrees has the {Helly property}, i.e.\ the intersection of all subtrees of a subtree representation of a clique is non-empty~\cite{Golumbic1980}, there is at least one common vertex $v_C$ for every clique $C$ in the host tree.
		A vertex in the host tree cannot belong to different maximal cliques since their union has to form a clique as well and hence the cliques would not be maximal. 
		Furthermore all subtrees intersecting in a vertex $v$ of the host tree build a clique $C$.
		If $C$ is not maximal, there is a maximal clique $C_{max}$ containing $C$. Contracting the path from $v$ to $v_{C_{max}}$ in the host tree does not increase the number of leaves. 
		Thus, if we choose an optimal representation with few vertices as possible, the claim follows. 
	\end{proof}
	
	We study the connection between the anchor width and the leafage of a chordal graph and show an upper bound of the anchor width.
	In the following, we identify a subtree of the host tree by its vertices.\\

	\begin{theorem}\label{thm:leafage}
		For a chordal graph $G$ with leafage $\ell = \ell(G)$ and $n$ vertices, the anchor width is at most $n^{\ell}$.
	\end{theorem}
	
	\begin{proof}
		Let $C = C_{max}$ be a maximal clique in the graph $G$. We consider an optimal subtree representation of $G$ such that the vertices of the host tree $T$ are in one-to-one correspondence with the maximal cliques of $G$ (cf. Lemma \ref{lem:1-1corresondence}). 
		Let $v_C$ be the vertex in the host tree which corresponds to the clique~$C$ of $G$. 
		From $v_C$ there is a unique path in the host tree to all $\ell$ leaves which we denote by $P_1, \ldots , P_{\ell}$.
		
		For a periphery set $M \subseteq P_C$, 
		the corresponding anchor set consists of those elements of the clique whose neighborhood contains $M$. 
		For every $w \in M$, there is a tree $T_w$ representing $w$ in the subtree configuration. 
		Since $C$ is a maximal clique, these trees $T_w$ do not contain $v_C$ since otherwise $w$ would belong to $C$.
		For every path $P_i$, we define a vertex $v_i$ representing $M$ on $P_i$ as follows:
		\[	v_i \in \argmin_{v \in T_{w_i} \cap P_i} dist(v,v_c),\]
		where $w_i$ is an element from the periphery such that
		\[
		w_i \in \argmax_{w \in M} \min_{v \in T_w \cap P_i} dist(v,v_c).
		\]
		So for every $w \in M$ such that $T_w \cap P_i \neq \emptyset$, we choose the closest vertex $v_w$ to $v_C$ on the path $P_i$ of the corresponding tree $T_w$. Among those vertices $\{v_w\}_w$, the vertex $v_i$ is the vertex with maximal distance to $v_C$.
		If there is no subtree $T_w$ of the periphery which has a non-empty intersection with the path $P_i$, we set $v_i = v_C$. 
		Note that the $v_i$'s are not necessarily distinct.
		
		Now the anchor set $A = A_G(M,C)$ consists exactly of all subtrees of the clique $C$, which contain all $v_1, \ldots, v_{\ell}$ and $v_C$. If there is no such subtree corresponding to an element of the clique, there is no corresponding anchor set to $M$ in $C$. 
		The anchor set $A$ is fully determined by the vertices $v_i$. 
		
		A chordal graph with $n$ vertices has at most $n$ maximal cliques. Hence the host tree has at most $n$ vertices which gives at most $n$ choices for every $v_i$.  
		In total we have at most $n^{\ell}$ choices for the tuple $(v_1, \ldots, v_{\ell})$ and hence at most $n^{\ell}$ different anchor sets. 
		This shows the upper bound for the anchor width. 
	\end{proof}
	
	Since interval graphs are the graphs with leafage at most $2$, it follows:
	\begin{corollary}
		The anchor width of interval graphs on $n$ vertices is at most $n^2$.
	\end{corollary}
	
	\begin{figure}[tb]
			\centering
		\includegraphics[scale=0.65]{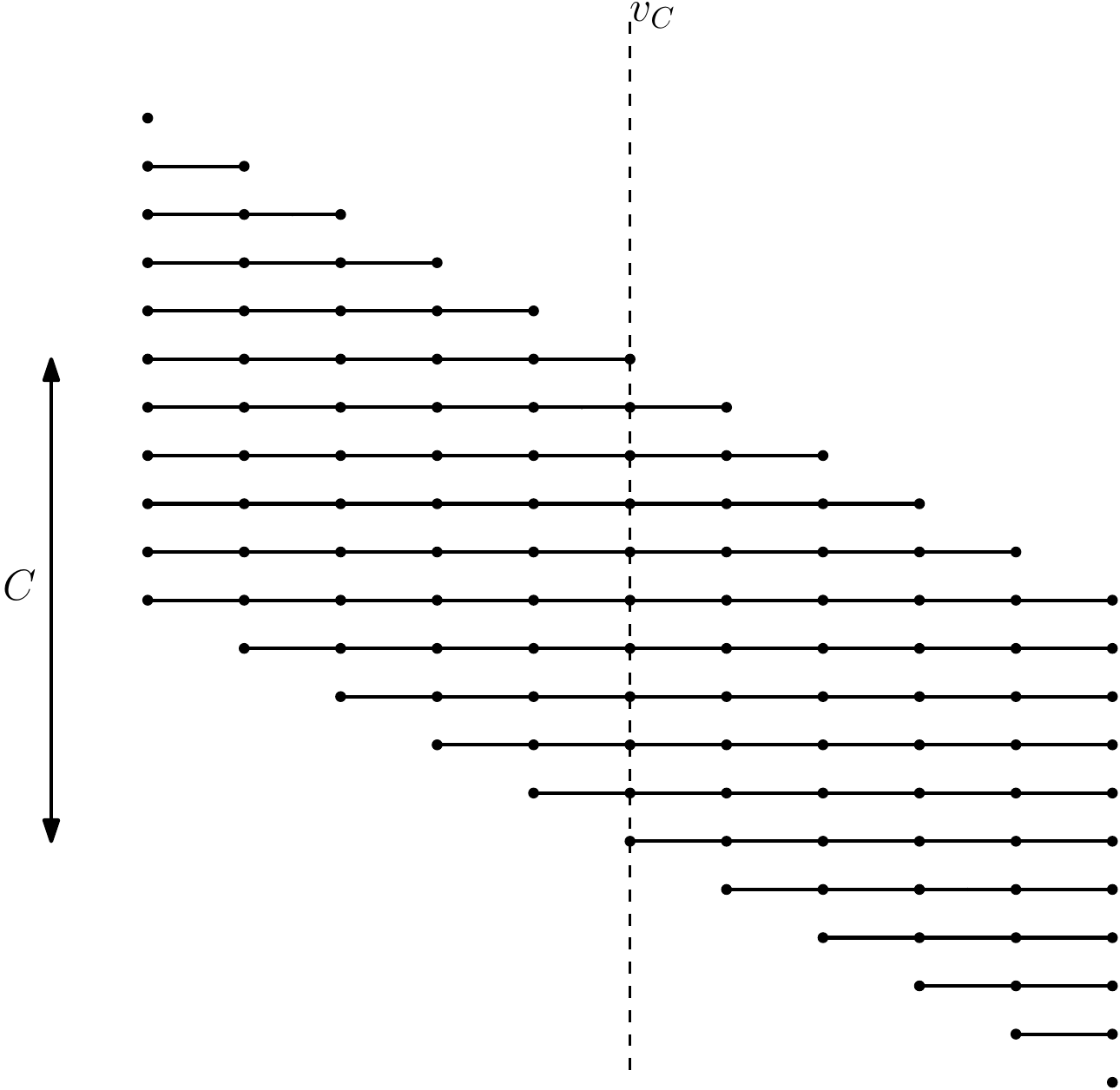}
		\caption{Construction of an interval graph with $21$ 
			vertices and a maximal clique of size $11$ 
			and $25$ anchor sets. }
		\label{fig:IntervalQuadraticBound}	
	\end{figure}

	The magnitude of the bound is optimal for interval graphs since there is an infinite family of interval graphs on $n = 4m +1$ vertices with a clique of size $2m+1$ which has at least $m^2 = \left( \frac{(n-1)}{4}\right)^2$ different anchor sets. 
	For the construction (see Figure \ref{fig:IntervalQuadraticBound}), we take the path $P$ on $2m+1$ vertices $v_{-m}, \ldots, v_0,\ldots,  v_m$ as a host tree. The subtrees corresponding to the clique $C$ are the $2m+1$ paths on the vertices
	\begin{align*}
	\{v_{-m}, \ldots, v_i\} \quad  &\text{ for } \  i = 0, \ldots, m \quad \text{ and } \\
		\{ v_i , \ldots,v_m\} \quad &\text{ for } \ i = -m+1, \ldots, 0.
	\end{align*} 
	The common intersection $v_C$ of the clique is the vertex $v_0$. 
	Furthermore we define the following subpaths, which are in the periphery of $C$:
	\begin{align*}
	\{v_{-m}, \ldots , v_i\} \quad &\text{ for } \ i = -m, \ldots , -1  \quad \text{and} \\
	\{v_i , \ldots, v_m \} \quad &\text{ for } \ i = 1 , \ldots, m.
	\end{align*}
	For every choice of $i \in \{-m ,\ldots, -1\}$ and $j \in \{1, \ldots, m \}$, we consider the two paths:
	\begin{align*}
		\{v_{-m}, \ldots , v_i\} \quad \text{ and } \quad \{v_j , \ldots ,v_m \}
	\end{align*}
	of the periphery. The anchor set corresponding to this two-element periphery 
	set consists of all paths in the host tree
	corresponding to a clique vertex which contain $v_i$ and $v_j$.
	For every choice of $i$ and $j$ these anchor sets differ. Hence there are at least $m^2$ anchor sets.

		Another interesting subclass are the chordal comparability graphs. For these graphs we show a linear bound on the anchor width.
	A graph $G = (V,E)$ is a \emph{comparability graph} if there is a poset $(V,\prec)$ such that two vertices $u,v \in V $ are adjacent in $G$ if and only if $u\prec v$ or $v \prec u $. \\
	
	\begin{proposition}\label{prop:comparability}
		The anchor width of a chordal comparability graph with $n$ vertices is at most~$2n$.
	\end{proposition}
	
		\begin{figure}[tb]
		\centering
		\includegraphics[page = 4,scale=0.65]{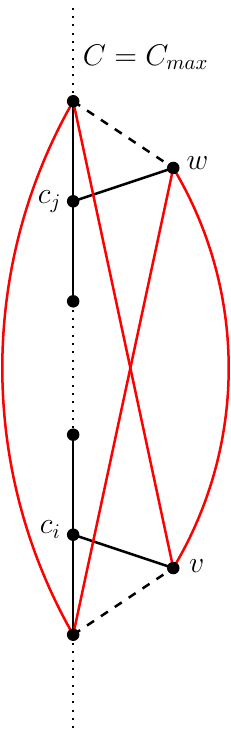} 
		\caption{Hasse diagram of a poset corresponding to a comparability graph with induced $C_4$ which gives a contradiction in the proof of Proposition \ref{prop:comparability};}
		\label{fig:ComparabilityGraph} 
	\end{figure}
	
	\begin{proof}
		Let $G$ be a chordal comparability graph with corresponding poset $(V, \prec)$. 
		Consider a maximal clique $C_{max} = \{c_1, \ldots ,c_m \}$ in $G$. 
		A clique in the graph corresponds to a chain in the poset. 
		Hence the maximal clique $C_{max}$ of size $m$ corresponds to a maximal chain $c_1 \prec \ldots \prec c_m$ of length $m-1$ in the poset.
		So whenever there is a vertex $v \notin C_{max}$, which is adjacent to $c_i \in C_{max}$ such that $v \prec c_i$ then $v $ is adjacent to $c_k$ for all $k \geq i$. Let $i$ be the minimal element of the clique such that $v \prec c_i$. 
		Since the clique is maximal, $i >1$ and $v$ is not comparable to $c_{i-1}$.
		Similar we get for every vertex $w \notin C_{max}$ which is connected to a vertex $c_j$ of the clique with $w  \succ c_j$ that $w$ is connected to all elements $c_k$ of the clique with $k \leq j$.
		Let $c_j$ be the maximal element of the clique connected to $w$, then $j <m$ and $c_{j+1}$ is not comparable to $w$. 
		Hence an anchor set in $C_{max}$ is a chain of the form $c_i \prec c_{i+1} \prec \ldots \prec c_{j-1} \prec c_j$.
		Assume there is an anchor set with $1< i <j < m$ and vertices $v$ and $w$ such that $v \prec c_i$ and $w \succ c_j$.
		Then $v$ and $w$ are connected by an edge since $w \succ c_j \succ c_i \succ v$ holds. And since $w$ and $c_{j+1}$ do not share an edge and analogously $v $ and $c_{i-1}$, we get an induced cycle of length $4$ which is not possible since the graph is chordal. In Figure \ref{fig:ComparabilityGraph} the poset is illustrated by its Hasse diagram and gives an illustration of the contradiction. 
		This shows that all anchor sets of $C_{max}$ are of the form $c_1 \prec \ldots \prec c_{j-1} \prec c_j$ for $j \leq m$ or $c_i \prec c_{i+1} \prec \ldots \prec  c_m$ for $i \geq 1$. We have at most $2m-1\leq 2n$ possibilities for those sets.
	\end{proof}

	\section{Discussion}
	In this paper we studied an algorithm for computing the neighborhood polynomial of chordal graphs, which is in general an \NP-hard problem. The runtime of the algorithm depends on the introduced parameter anchor width. If the anchor width of a subclass of chordal graphs is bounded, we have a polynomial-time algorithm to compute the neighborhood polynomial. 
	In Section~\ref{sec:AnchorWidth} we investigated some subclasses and showed that the anchor width is bounded for chordal graphs with bounded leafage and chordal comparability graphs. 
	Furthermore we showed that the anchor width is not bounded for split graphs. 
	It would be interesting to get further subclasses of chordal graphs with bounded anchor width. 
	It might be possible to give an upper bound for the anchor width using the asteroidal number. In \cite{LinMcKeeWest1998} it is shown that the leafage is an upper bound for the asteroidal number for all chordal graphs and they coincide for chordal graphs whose host tree is a subdivision of $K_{1,n}$ as shown in \cite{Prisner92}.
	Furthermore an infinite family of graphs similar to the one for interval graphs, which shows that the magnitude of the upper bound is optimal, would be interesting. 
	
	On top of that there might be other problems on chordal graphs which are hard in general  but polynomial solvable on those subclasses with bounded anchor width. 
	One natural candidate would be graph isomorphism, which is known to be graph-isomorphism-complete on general chordal graphs \cite{LuekerBooth1979} but has recently been shown to be solvable in polynomial time for chordal graphs of bounded leafage \cite{Zeman21}.
	 
	\bibliographystyle{abbrvnat}

	\bibliography{literature}

\begin{thebibliography}{16}
\providecommand{\natexlab}[1]{#1}
\providecommand{\url}[1]{\texttt{#1}}
\expandafter\ifx\csname urlstyle\endcsname\relax
  \providecommand{\doi}[1]{doi: #1}\else
  \providecommand{\doi}{doi: \begingroup \urlstyle{rm}\Url}\fi

\bibitem[Alipour and Tittmann(2021)]{AlipourTittmann2021}
M.~Alipour and P.~Tittmann.
\newblock Graph operations and neighborhood polynomials.
\newblock \emph{Discussiones Mathematicae Graph Theory}, 41\penalty0
  (3):\penalty0 697--711, 2021.
\newblock \doi{10.7151/dmgt.2347}.

\bibitem[Arvind et~al.(2021)Arvind, Nedela, Ponomarenko, and Zeman]{Zeman21}
V.~Arvind, R.~Nedela, I.~Ponomarenko, and P.~Zeman.
\newblock Testing isomorphism of chordal graphs of bounded leafage is
  fixed-parameter tractable.
\newblock \href{https://arxiv.org/abs/2107.10689}{arXiv:2107.10689}, 2021.

\bibitem[Bergold et~al.(2021)Bergold, Hochst{\"{a}}ttler, and Mayer]{WADS}
H.~Bergold, W.~Hochst{\"{a}}ttler, and U.~Mayer.
\newblock The neighborhood polynomial of chordal graphs.
\newblock In \emph{Algorithms and Data Structures - 17th International
  Symposium, {WADS} 2021}, volume 12808 of \emph{LNCS}, pages 158--171.
  Springer, 2021.
\newblock \doi{10.1007/978-3-030-83508-8_12}.

\bibitem[Bertossi(1984)]{Bertossi1984}
A.~A. Bertossi.
\newblock Dominating sets for split and bipartite graphs.
\newblock \emph{Information Processing Letters}, 19\penalty0 (1):\penalty0
  37--40, 1984.
\newblock \doi{10.1016/0020-0190(84)90126-1}.

\bibitem[Booth and Johnson(1982)]{BoothJohnson1982}
K.~S. Booth and J.~H. Johnson.
\newblock Dominating sets in chordal graphs.
\newblock \emph{SIAM Journal on Computing}, 11\penalty0 (1):\penalty0 191--199,
  1982.
\newblock \doi{10.1137/0211015}.

\bibitem[Brown and Nowakowski(2008)]{BrownNowakowski2008}
J.~I. Brown and R.~J. Nowakowski.
\newblock The neighbourhood polynomial of a graph.
\newblock \emph{Australasian Journal of Combinatorics}, 42:\penalty0 55--68,
  2008.

\bibitem[Corneil et~al.(1985)Corneil, Perl, and
  Stewart]{CorneilPerlStewart1985}
D.~G. Corneil, Y.~Perl, and L.~K. Stewart.
\newblock A linear recognition algorithm for cographs.
\newblock \emph{SIAM Journal on Computing}, 14\penalty0 (4):\penalty0 926--934,
  1985.
\newblock \doi{10.1137/0214065}.

\bibitem[Day(2017)]{Day2017}
D.~Day.
\newblock On the neighbourhood polynomial, 2017.
\newblock URL
  \url{https://dalspace.library.dal.ca/bitstream/handle/10222/72816/Day-Dylan-MSc-MATH-March-2017.pdf}.
\newblock Last access on 24/04/2022.

\bibitem[Gavril(1974)]{Gavril1974}
F.~Gavril.
\newblock The intersection graphs of subtrees in trees are exactly the chordal
  graphs.
\newblock \emph{Journal of Combinatorial Theory, Series B}, 16\penalty0
  (1):\penalty0 47--56, 1974.
\newblock \doi{10.1016/0095-8956}.

\bibitem[Golumbic(1980)]{Golumbic1980}
M.~C. Golumbic.
\newblock \emph{Algorithmic Graph Theory and Perfect Graphs}.
\newblock 1980.
\newblock ISBN 978-0-12-289260-8.
\newblock \doi{10.1016/C2013-0-10739-8}.

\bibitem[Habib and Stacho(2009)]{HabibStacho2009}
M.~Habib and J.~Stacho.
\newblock Polynomial-time algorithm for the leafage of chordal graphs.
\newblock In \emph{Algorithms - {ESA} 2009, 17th Annual European Symposium},
  volume 5757 of \emph{LNCS}, pages 290--300. Springer, 2009.
\newblock \doi{10.1007/978-3-642-04128-0_27}.

\bibitem[Heinrich and Tittmann(2018)]{HeinrichTittmann2018}
I.~Heinrich and P.~Tittmann.
\newblock Neighborhood and domination polynomials of graphs.
\newblock \emph{Graphs and Combinatorics}, 34:\penalty0 1203--1216, 2018.
\newblock \doi{10.1007/s00373-018-1968-7}.

\bibitem[Lin et~al.(1998)Lin, McKee, and West]{LinMcKeeWest1998}
I.~Lin, T.~A. McKee, and D.~B. West.
\newblock The leafage of a chordal graph.
\newblock \emph{Discussiones Mathematicae Graph Theory}, 18\penalty0
  (1):\penalty0 23--48, 1998.
\newblock \doi{10.7151/dmgt.1061}.

\bibitem[Lov\'asz(1978)]{Lovasz1978}
L.~Lov\'asz.
\newblock Kneser's conjecture, chromatic number, and homotopy.
\newblock \emph{Journal of Combinatorial Theory, Series A}, 25\penalty0
  (3):\penalty0 319--324, 1978.
\newblock \doi{10.1016/0097-3165(78)90022-5}.

\bibitem[Lueker and Booth(1979)]{LuekerBooth1979}
G.~S. Lueker and K.~S. Booth.
\newblock A linear time algorithm for deciding interval graph isomorphism.
\newblock \emph{Journal of the ACM}, 26\penalty0 (2):\penalty0 183--195, 1979.
\newblock \doi{10.1145/322123.322125}.

\bibitem[Prisner(1992)]{Prisner92}
E.~Prisner.
\newblock Representing triangulated graphs in stars.
\newblock \emph{Abhandlungen aus dem Mathematischen Seminar der Universit\"{a}t
  Hamburg}, 62:\penalty0 29--41, 1992.
\newblock \doi{10.1007/BF02941616}.

\end{thebibliography}
	\label{sec:biblio}
\end{document}